\documentclass[twoside,openright,11pt,a4paper,epsf]{article}
\usepackage{latexsym}
\usepackage[utf8]{inputenc}
\usepackage{amsmath}
\usepackage{amsthm}
\usepackage{amsfonts}
\usepackage{amssymb}
\usepackage{epsfig}
\usepackage{nicefrac}
\usepackage[all]{xy}
\usepackage{graphicx}

\newcommand{\color}[6]{}
\newcommand{\R}{\mathbb{R}}
\newcommand{\C}{\mathbb{C}}
\newcommand{\D}{\mathbb{D}}
\renewcommand{\P}{\mathbb{P}}
\newcommand{\N}{\mathbb{N}}

\newcommand{\Q}{\mathbb{Q}}

\newcommand{\cb}{\mathbb{\mathcal B}}
\newcommand{\ca}{\mathcal A}
\newcommand{\ce}{\mathcal E}

\newcommand{\cp}{\mathcal{P}}
\newcommand{\cs}{\mathcal{S}}
\newcommand{\ct}{\mathcal{T}}

\newcommand{\nbd}{neighbourhood }

\newcommand{\Om}{\Omega}

\newcommand{\priv}{\backslash}
\newcommand{\lra}{\longrightarrow}
\newcommand{\hra}{\hookrightarrow}

\newcommand{\om}{\omega}
\newcommand{\eps}{\varepsilon}
\renewcommand{\phi}{\varphi}

\newcommand{\vol}{\text{Vol}\,}
\newcommand{\conv}{\text{Conv}\,}

\newcommand{\cqfd}{\hfill $\square$ \vspace{0.1cm}\\ }
\newcommand{\sbull}{{\tiny $\bullet$ }}
\newcommand{\ds}{\displaystyle}

\newcommand{\pd}{\textnormal{\small PD}}
\newcommand{\omhra}{\overset{\om}{\hra}}
\newcommand{\bw}{{\bf w}}

\newtheorem{definition}{Definition}[section]
\newtheorem{thm}{Theorem}
\newtheorem*{thm*}{Theorem}
\newtheorem{prop}[definition]{Proposition}
\newtheorem{lemma}[definition]{Lemma}

\newtheorem{cor}[definition]{Corollary}
\newtheorem{rk}[definition]{Remark}

\newtheorem{conj}{Conjecture}

\title{ \vspace*{0cm} Packing stability for symplectic $4$-manifolds.}
\author{O. Buse\footnote{Partially supported by NSF grant  DMS-1211244.}, R. Hind and E. Opshtein\footnote{Partially supported by ANR project "hameo" ANR-116JS01-010-01.}.}
\begin{document}
\maketitle
\begin{abstract}
The packing stability  in symplectic geometry was first noticed by Biran \cite{biran5}: the symplectic obstructions to embed several balls into a manifold disappear when their size is small enough. This phenomenon is known to hold for all closed manifolds with rational symplectic class (see \cite{biran2} for the $4$-dimensional case, and \cite{buhi,buhi2} for higher dimensions), as well as for all ellipsoids \cite{buhi2}.

In this note, we show that packing stability holds for all closed, and several open, symplectic $4$-manifolds.
\end{abstract}
\section{Introduction}
In \cite{biran5,biran2}, Biran discovered the packing stability phenomenon: in some symplectic $4$-manifolds, the symplectic obstructions to pack $N$ identical balls disappear when $N$ becomes large enough. He later generalized this result to every {\it closed} symplectic $4$-manifold with rational symplectic class ($[\om]\in H^2(M,\Q)$).

In this paper we generalize these results in several directions, by applying the singular inflation technique developed in \cite{moi6,mcop}. We generalize Biran's packing stability to all symplectic $4$-manifolds, and also to a class of open symplectic $4$-manifolds including ellipsoids and domains we call pseudo-balls.
For these manifolds, we establish not just packing stability but strong packing stability, which we define now. The definition is in the spirit of Biran's first results on the subject \cite{biran5}. Throughout $B(\lambda)$ will denote an open ball of capacity $\lambda$ (see Definition \ref{defnpseudo}) and we will write $U \omhra V$ for a symplectic embedding.

\begin{definition} A symplectic manifold $X$ has property $\cp(\lambda)$ if for all collection of real numbers $(\lambda_i)$ with $\lambda_i<\lambda\;\forall i$,  $\sqcup B(\lambda_i)\omhra X$ if and only if the volume obstruction $\frac{1}{2}\sum \lambda_i^2 \le \mathrm{Vol}(X)$ is satisfied. We say that $X$ has the strong packing stability if it satisfies $\cp(\lambda)$ for some $\lambda$. We then define
$$
\Lambda(X):=\inf\{\lambda \;|\; X \text{ verifies } \cp(\lambda)\}.
$$

If there exists a $\mu$ such that for all $\lambda< \mu$ we have $\sqcup B(\lambda)\omhra X$ if and only if the volume obstruction is satisfied, then we say that $X$ has packing stability. (In other words, we only consider embeddings by equally sized balls.)
\end{definition}

\begin{rk} Since the volume inequality used is not strict, this definition even of packing stability is stronger than those used previously (for example in \cite{buhi}) which ask for embeddings only if the union of balls has strictly smaller volume. Of course, to obtain embeddings in the case of equal volume, it is necessary for us to work with open balls. Such embeddings were called very full fillings in \cite{lamcsc}, section $7$.
\end{rk}

Our main result is the following.

\begin{thm}\label{thm:irrstabclosed}
All closed $4$-dimensional symplectic manifolds have strong packing stability.
 \end{thm}

To prove this theorem we will first need to establish strong packing stability for some open symplectic manifolds, and these results are interesting in themselves. To define these domains we work in $\C^2$ with its standard symplectic form.

\begin{definition}\label{defnpseudo} Let $U \subset \C^2$ be the interior of a domain of the form $\{(z,w)\in \C^2\; |\; (\pi |z|^2,\pi |w|^2)\in P\}$ where $P$ is a subset of the first quadrant in $\R^2$.

If $P$ is the convex hull $P=\conv\langle (0,0), (a,0), (0,b) \rangle$ then we say $U$ is a symplectic ellipsoid $E(a,b)$. A ball of capacity $a$ is an ellipsoid $B(a)=E(a,a)$, and we write $\tau E(a,b)$ for $E(\tau a, \tau b)$.

If $P=\conv\langle (0,0), (0,a), (b,0), (\alpha,\beta) \rangle$ for some $a>\alpha$, $b>\beta$ and $a,b< \alpha+\beta$ then we say that $U$ is a pseudo-ball $T(a,b,\alpha,\beta)$.
\end{definition}

Now we can state the following.

\begin{thm}\label{thm:irrstabell}
The $4$-dimensional ellipsoids have the strong packing stability. Moreover, the function $\Lambda(E(1,a))$ is locally bounded on $\R^*$.
\end{thm}

Note that since any rational symplectic manifold has full packing by an ellipsoid \cite{moi3}, this theorem clearly implies Biran's original result.
\begin{cor} \label{cor1} All closed  symplectic $4$-manifolds with rational symplectic class have the packing stability.
\end{cor}

Strong packing stability for ellipsoids does not immediately imply packing stability for all closed manifolds, however, because the following remains a conjecture.

\begin{conj} \label{conj1}
Every (closed) $4$-dimensional symplectic manifold is fully packed by one ellipsoid.
\end{conj}

Nevertheless we will show that any symplectic $4$ manifold can be decomposed (up to a subset of volume $0$) into finitely many open symplectic manifolds, each of which has a Hamiltonian toric free action with convex associated polytope, namely ellipsoids and pseudo-balls. Theorem \ref{thm:irrstabclosed} will then follow from Theorem \ref{thm:irrstabell} and the following strong packing stability for pseudo-balls.

\begin{thm}\label{thm:irrstabpseudo}
The pseudo-balls have the strong packing stability. Moreover, the function $\Lambda(T(a,b,\alpha,\beta))$ is locally bounded on $\{\alpha<a<\alpha+\beta, \beta<b<\alpha+\beta\}\subset \R^4$.
\end{thm}

\paragraph{Related results.}
In \cite{buhi,buhi2}, the first two authors proved that rational symplectic manifolds in all dimensions have packing stability. It remains an open question whether higher dimensional rational manifolds have strong packing stability. The proof argued as in Corollary \ref{cor1} by first showing that all ellipsoids have packing stability and then remarking that any {\it rational} symplectic manifold admits a volume filling embedding from an ellipsoid. The ellipsoid packing stability ultimately relied on an embedding result for $4$-dimensional ellipsoids.

\begin{thm*} \label{ellemb} [Buse-Hind] There exists a continuous function $f:[1,\infty) \to [1,\infty)$ such that if $b>f(a)$ then there exists a symplectic embedding $\lambda E(1,b) \to E(1,a)$ if any only if the volume condition is satisfied.
\end{thm*}

This immediately implies ellipsoid packing stability in four dimensions since an ellipsoid $E(1,k)$ can be fully filled by $k$ balls. The proof of the above theorem in \cite{buhi2} relied on Embedded Contact Homology, although the theorem in fact also follows from the strong packing stability we establish here. Indeed, work of McDuff says that an ellipsoid embedding in dimension $4$ exists whenever there exists an embedding of a union of balls of appropriate sizes. It would be interesting to compare the corresponding bounds on the function $f$.

Recent work of Latschev, McDuff and Schlenk in \cite{lamcsc} implies packing stability of $4$-dimensional tori with linear symplectic forms (most of which are irrational). Indeed, if $T$ is a linear $4$-dimensional torus not a product of $2$-dimensional tori of equal area then we have an embedding $\sqcup B(\lambda_i)\omhra X$ whenever the volume obstruction $\frac{1}{2}\sum \lambda_i^2 < \mathrm{Vol}(X)$ is satisfied. For the product $T(\mu, \mu)$ of two tori of area $\mu$ this remains true given the additional condition that $\lambda_i<\mu$ for all $i$. However whether or not there exist very full fillings of tori under just these hypotheses remains an open question.



\paragraph{Organization of the paper and description of methods.}
As mentioned
above,
the main ingredient for our arguments is to show that two types of open symplectic manifolds, an ellipsoid and a pseudo-ball,
have strong packing stability. The conclusion of the proof of packing stability for ellipsoids (Theorem \ref{thm:irrstabell}) is found in section \ref{section33} and explain how the case of the pseudo-ball (Theorem \ref{thm:irrstabpseudo}) follows along the same lines in section \ref{packpseudo}.

Both these results are consequences of two technical lemmas, which we deal with first. One is a basic result on packing stability of blow-ups of $\P^2$,
proved in section \ref{firstlemma},
which serves as a building block for all our other packing stability results.

\begin{lemma}\label{le:ratstabpack} All symplectic forms on the $p$-fold blow-up $\hat \P^2_p$ of $\P^2$ have the strong packing stability. Moreover  $\Lambda(\hat \om)$ is bounded from below by a quantity which depends only on the volume of $(\P^2_p,\hat \om)$.
\end{lemma}


The second lemma,
of independent interest, allows us to isotope any packing by ellipsoids into a good position with respect to a given symplectic curve, assuming only the necessary area requirements.

\begin{lemma}\label{le:directed}
Let $\ds\sqcup_{i=1}^k E(a_i,b_i)\sqcup_{i=1}^{k'} E(a_i',b_i')\sqcup_{i=1}^{k''} E(a_i'',b_i'')\sqcup_{i=1}^{k^{(3)}} E(a_i^{(3)},b_i^{(3)})\overset{\phi}{\hra} \P^2$ be a symplectic embedding, which is the restriction of an embedding of the disjoint union of closed ellipsoids, and let $C$ an immersed,
possibly reducible,
symplectic curve with positive self-intersections only.
We denote by $E_i$ the ellipsoid $E(a_i,b_i)$, by $E'_i$ the ellipsoid $E(a_i',b_i')$ and so on.

Let us assign a component of $C$ to each ellipsoid $E(a_i,b_i)$ and to each $E(a_i',b_i')$. Suppose we can also assign a self-intersection point of $C$ to each ellipsoid $E(a_i'',b_i'')$ together with a different branch for the two axes of the ellipsoids. Assume that each component of $C$ has area strictly greater than the sum of all $a_p$, $b_q'$, $a_r''$ and $b_s''$ for which $E(a_p,b_p)$, $E(a_q',b_q')$, the first axis of $E(a_r'',b_r'')$ and the second axis of $E(a_s'',b_s'')$ are assigned to it.

Then $C$ is symplectically isotopic to a $C'$ with the following properties:
$$
\begin{array}{l}
{E}_i\cap C'=\phi(\{w=0\})=\phi(\D(a_i)\times \{0\}),\\
{E}_i'\cap C'= \phi(\{z=0\})= \phi(\{0\}\times \D(b_i'))\\
{E}_i''\cap C'= \phi(\{z=0\}\cup \{w=0\})=\phi(\D(a_i'')\times \{0\}\cup \{0\}\times \D(b_i'')) \\
{E}_i^{(3)}\cap C'=\emptyset
\end{array}
$$

\end{lemma}

We prove this lemma in section \ref{secondlemma}, after recalling some results of McDuff relating ellipsoid and ball packings in section \ref{background}.

\begin{rk}
Although this lemma definitely uses the fact that $C$ is $J$-holomorphic for an $\om$-compatible $J$, it does not seem to be a purely pseudo-holomorphic statement. The lemma applies in cases when our $J$-curves have very negative index, and can therefore be found only in very large-dimensional families  of almost complex structure. This makes
the task of finding isotopies of these curves  {\it via} pseudo-holomorphic methods rather hopeless.
The proof relies instead on the singular inflation technique, and what might be surprising is that it shows that this statement belongs to the {\it soft} side of symplectic geometry.
\end{rk}

Given these preliminaries, Theorem \ref{thm:irrstabclosed} is put together in section \ref{mainthm}. In section \ref{mainone} we describe the  decomposition of (irrational) symplectic manifolds into a finite number of ellipsoids and pseudo-balls. Applying the packing stability for these open manifolds, we complete the proof of Theorem \ref{thm:irrstabclosed} in section \ref{maintwo}.

\section{Packing stability of rational surfaces}\label{firstlemma}
The aim of this section is to explain lemma \ref{le:ratstabpack}. We recall the following result by Biran \cite{biran5}.
Let $(\hat\P^2_p,\hat \om)$ be a symplectic blow-up of $\P^2$, and call $\Om:=[\hat \om]$ the cohomology class of $\hat \om$. We normalize so that lines in $\P^2$ have area $1$. Define
$$
d_\Om:= \inf\left\{ \frac{\Om(B)}{c_1(B)},\; \text{ for } B^2\geq 0,\; \Om(B)>0,\; c_1(B)\geq 2 \right\}.
$$
\begin{thm*}[Biran]  The quantity $d_\Om$ bounds $\Lambda(\hat \P^2_p,\hat \om)$ from below.
\end{thm*}
Of course, this theorem is interesting only when $d_\Om>0$. Lemma \ref{le:ratstabpack} will therefore follow if we can bound $d_\Om$ from below by a quantity which depends only on the volume of $(\hat \P^2_p,\hat \om)$. This is precisely the content of the following lemma.
\begin{lemma} \label{le:dom} Let $\hat \om$ be a symplectic form on $\hat \P^2_p$ obtained by blowing-up $p$ balls of sizes $(\lambda_1,\dots,\lambda_p)$.  Denoting by $\kappa:=\sqrt{\sum \lambda_i^2}$, we have
$$
d_\Om\geq \frac{1-\kappa}{3+\sqrt p} = \frac{1-\sqrt{1-2\mathrm{Vol}(\hat\P^2_p)}}{3+\sqrt p}.
$$
\end{lemma}

Note that to be able to blow-up balls of these sizes we automatically have $\kappa <1$.

\noindent{\it Proof:} Let $B\in H_2(\hat \P^2_p)$, decomposed in the following way:
$$
B=kL-\sum m_iE_i,
$$
where $L,E_1,\dots,E_p$ is the standard basis of $H_2(\hat \P^2_p)$, and $k,m_i$ are integers.

The condition $B^2\geq 0$ implies that $k^2-\sum m_i^2\geq 0$, hence
\begin{equation}\label{eq:Bsq}
\sum m_i^2\leq k^2.
\end{equation}

Hence by Cauchy-Schwartz inequality, for a class with $B^2 \ge 0$ we have
$$
\left|\sum m_i\lambda_i\right|\leq \sqrt{\sum m_i^2}\sqrt{\sum \lambda_i^2} =  \kappa\sqrt{\sum m_i^2} < \sqrt{\sum m_i^2} \leq |k|.
$$

Now $\Om(B)=k-\sum m_i\lambda_i$.
We therefore see that under the conditions $B^2>0$ and $\Om(B)>0$ we have $\Om(B)\leq k+|k|$ and hence $k>0$ and, by the same Cauchy-Schwarz inequality
$$\Om(B)>k(1-\kappa).$$

Finally, for a class with $B^2\geq 0$ and $\Om(B)>0$,
$$
c_1(B)=3k-\sum m_i\leq 3k+\sqrt p\sqrt{\sum m_i^2}\leq 3k+\sqrt p k.
$$
Putting these inequalities together, we get:
$$
d_\Om\geq \frac{(1-\kappa)k}{(3+\sqrt p)k}=\frac{1-\kappa}{3+\sqrt p}.
$$
\cqfd

\section{Packing stability of ellipsoids}
\subsection{Background results on ellipsoid embeddings}\label{background}
We recall two facts on embedding of ellipsoids. The first one, stated in \cite{mcduff4}, relates ellipsoid embeddings to ball packings.

\begin{thm}[McDuff, \cite{mcduff4}]\label{mcduffthm} For each $a\in\Q^+$, there is an in integer $p(a)$ and weights $\bw(a):=(w_1(a),\dots,w_{p(a)}(a))$, such that for all rational symplectic $4$-manifolds $(X,\om)$, $\tau E(1,a)\omhra X$ if and only if $\sqcup \tau B(w_i)\omhra X$.
\end{thm}

The second one follows readily from the proof of the previous theorem, in particular the construction of the $w_i(a)$.


\begin{lemma}\label{le:stabw}
Fix $a\in \R$. There exists $\eps_0$ (that depends on $a$), such that for all $\eps<\eps_0$ and $a'$ with $|a'-a|<\eps$, $p'(a)>p(a)$ and $\bw(a')$ verifies the following:
$$
\begin{array}{ll}
|w_i(a')-w_i(a)| < \eps \hspace*{2cm} & \forall i\leq p,\\
|w_i(a')|<\eps & \forall i>p.
 \end{array}
$$
\end{lemma}

We will need to apply not only Theorem \ref{mcduffthm}, but also the blowing up procedure in a neighborhood of an ellipsoid which it relies on. Therefore we briefly recall this procedure now, following \cite{mcduff4}, section $3$.

First of all, the number of steps in a continued fraction expansion gives us a function $p:\Q^+ \to \N$. Then McDuff shows that corresponding to each rational ellipsoid $E=E(a,b)$ and $p$-tuple $\delta \in \R^p$ of small (admissible) numbers, where $p=p(b/a)$, there is a sequence of symplectic spheres $S_1, \dots ,S_p$ which are $J$-holomorphic for a tame almost-complex structure on a suitable $p$-fold blow-up of $\C^2$. The complement of $\cup S_j$ in a small neighborhood $V$ of  $\cup S_j$ is symplectomorphic to a neighborhood of the boundary in a carefully chosen but arbitrarily small domain containing $E \subset \C^2$. In other words, an arbitrarily small neighborhood of $E$ can be replaced by a configuration of $J$-holomorphic curves, and conversely, by a version of Weinstein's Theorem, given any such configuration of curves where the components have the required areas and relative intersections, we can remove a small neighborhood and replace it with a neighborhood of our ellipsoid $E$.

The sequence of curves $S_j$ is constructed recursively starting with a neighborhood $U_0 \subset \C^2$ of $E$ which is defined in terms of the choice of $\delta$. The precise algorithm for the blowing-up can be found in \cite{mcduff4}, here we just outline briefly how the singular set arises.

{\it Step $1$.} At the first stage we blow-up a ball of size $k_1=aw_1(b/a)+\delta_1 = a + \delta_1$ inside $U_0$. The ball intersects the short axis, say $\{z_2=0\}$, of $E$ along a complex line and so the blow-up $U_1$ contains an exceptional divisor $E_1$ and the proper transform $C_1$ of $\{z_2=0\}$.

{\it Step $2$.} Now we blow-up a ball of size $k_2 = a w_2(b/a) + \delta_2$ inside $U_1$. This ball will intersect $E_1$ and $\{z_1=0\}$ along complex lines while avoiding $C_1$. Thus the blow-up $U_2$ will contain an exceptional divisor $E_2$, the proper transform of $E_1$, the proper transform $C_2$ of $\{z_1=0\}$ and still $C_1$.

{\it Step $i$.} Again we blow-up a ball of a certain size $k_i = a w_i(b/a)$ to produce a blow-up $U_i$ of $U_{i-1}$ with an exceptional divisor $E_i$. This ball will intersect exactly two of the curves we have considered previously, that is, proper transforms of $C_1$, $C_2$ or $E_j$ for $j<i$, and will intersect these along complex lines.

The construction is performed so that the volume of $U_p$ is of order $\delta$ and all points originally in the ellipsoid $E$ eventually lie in one of the balls to be blown-up. The curve $S_j$ is the proper transform of $E_j$. Thus its self-intersection is $-1$ minus the number of subsequent balls which intersect the transform of $E_j$. The construction is such that $S_p=E_p$ is the only curve with intersection $-1$ and the $S_j$ for $j<p$ all have area of order $\delta$.

The proper transforms of $C_1$ and $C_2$ will each intersect exactly one of the $S_j$, corresponding to the last ball which intersects the transform of the planes $\{z_2=0\}$ and $\{z_1=0\}$ respectively. When we reverse the construction we may assume that any given holomorphic curve which intersects this $S_j$ in a single point transversally is blown-down to a curve intersecting our neighborhood of $E$ in the corresponding axis.

\subsection{Ellipsoid embeddings in good positions relative to a given curve}\label{secondlemma}

In this section we give the proof of Lemma \ref{le:directed}, and follow that notation. First note that since we start with embeddings of closed ellipsoids, and since the area inequalities for components of $C$ are strict, the embeddings extend to disjoint rational ellipsoids which still satisfy the inequalities. Hence we may assume that all of our ellipsoids are rational. Moreover when choosing parameters $\delta$ as in section \ref{background} we will assume that the embeddings of disjoint closed ellipsoids extend to embeddings of disjoint neighborhoods of size $\delta$.
As in section \ref{firstlemma} we normalize so that lines in $\P^2$ have area $1$.



We need to show that $C$ is symplectically isotopic to a curve $C'$ in a special position with respect to the ellipsoids. Since in $\P^2$, ellipsoid packings are all isotopic, see \cite{mcop}, we can alternatively prove that  there exists a packing by the same ellipsoids which is in the desired position with respect to $C$ (we will say for short that our packing is {\it directed} by the curve $C$).

Our hypotheses on the existence of intersection points together with Darboux's theorem implies the existence at least of directed symplectic embeddings of the union $\sqcup \tau E_i \sqcup \tau E'_i \sqcup \tau E''_i \sqcup \tau E^{(3)}_i$ for $\tau$ sufficiently small.

We then blow-up the scaled ellipsoids (or rather the correct \nbd of the ellipsoids), as described in section \ref{background}. This provides a symplectic form $\hat \om_\tau$ on the $P$-fold blow-up $\hat \P^2_P$ of $\P^2$, where
$$
P=\sum_{i=1}^k p\big(\frac{b_i}{a_i}\big) + \sum_{i=1}^{k'} p\big(\frac{b'_i}{a'_i}\big)
+ \sum_{i=1}^{k''} p\big(\frac{b''_i}{a''_i}\big) + \sum_{i=1}^{k^{(3)}} p\big(\frac{b^{(3)}_i}{a^{(3)}_i}\big),
$$
and $p$ is the integer function that appears in Theorem \ref{mcduffthm} above. Write $p_i=p(\frac{b_i}{a_i})$, $p'_i=p(\frac{b'_i}{a'_i})$, $p''_i=p(\frac{b''_i}{a''_i})$ and $p^{(3)}_i=p(\frac{b^{(3)}_i}{a^{(3)}_i})$.

The symplectic form comes with a singular set $\cs$ (in the sense of \cite{mcop}): disjoint bunches of embedded negative spheres $\cs_i$, $\cs'_i$, $\cs''_i$ and $\cs^{(3)}_i$ corresponding to each ellipsoid as in section \ref{background}. We denote the spheres in $\cs_i$ by $S_{i,j}$, the spheres in $\cs'_i$ by $S'_{i,j}$ and so on. Also, we denote the exceptional classes corresponding to the $\tau E_i$ by $[F_{i,j}]$, that is, $[F_{i,j}]$ is the class of the exceptional divisor arising from the $j$th blow-up of a neighborhood of the embedded $\tau E_i$. Similarly define $[F'_{i,j}]$, $[F''_{i,j}]$ and $[F^{(3)}_{i,j}]$. Then with respect to the form $\hat \om_\tau$ we may assume that the class $[F_{i,j}]$ has area $\tau k_{i,j} = \tau(a_i w_i(b_i/a_i) + \delta_{i,j})$, $[F'_{i,j}]$ has area $\tau k'_{i,j}$ and so on for the other ellipsoids.


The blow-up $\hat \P^2_P$ also contains the strict transform of $C$, denoted $\hat C$. As $C$ intersects the ellipsoids along their axes (if at all) we recall again from the previous section that this implies $\hat C$ intersects transversally, $\om$-orthogonally, exactly one of the spheres $S_{i,j}$, exactly one of the spheres $S'_{i,j}$, two of the spheres $S''_{i,j}$ but avoids the $S^{(3)}_{i,j}$. As a consequence $\ct:=\cs\cup C$ is a singular set in the sense of \cite{mcop}.



Let $L$ denote the class of a line in $\P^2$, and also of the same line in $\hat \P^2_P$.

\begin{lemma}\label{le:directed-blow-up}
There exists a symplectic packing of $\sqcup E_i \sqcup E'_i \sqcup E''_i \sqcup E^{(3)}_i$, directed by $C$, if there exists $N\in \N$, $\delta>0$ and a connected embedded symplectic curve $Q$ with $Q^2>0$ in the class
$$
A=N(L-(1+\delta)\sum_{i,j} k_{i,j}[F_{i,j}] -(1+\delta)\sum_{i,j} k'_{i,j}[F'_{i,j}] $$ $$-(1+\delta)\sum_{i,j} k''_{i,j}[F''_{i,j}] -(1+\delta)\sum_{i,j} k^{(3)}_{i,j}[F^{(3)}_{i,j}]),
$$
which intersects $\ct$ positively and transversally.
\end{lemma}

\begin{proof}
As $Q^2 \ge 0$ we can inflate by adding a tubular neighborhood of $Q$ with fibers of any area $r>0$. The result will be a symplectic form Poincar\'{e} dual to $$
(1+rN)L-\sum_{i,j} (\tau + Nr(1+\delta))k_{i,j}[F_{i,j}] -\sum_{i,j} (\tau + Nr(1+\delta))k'_{i,j}[F'_{i,j}] $$ $$-\sum_{i,j} (\tau + Nr(1+\delta))k''_{i,j}[F''_{i,j}] -\sum_{i,j} (\tau + Nr(1+\delta))k^{(3)}_{i,j}[F^{(3)}_{i,j}].
$$

We choose $r$ such that $1+rN = \tau + rN(1+\delta)$ or equivalently $r=\frac{1-\tau}{N\delta}$. Then dividing the inflated form by $1+rN$, the singular set will have areas corresponding to ellipsoids of the correct sizes and $L$ will have area $1$. Furthermore $\ct$ remains symplectic with $\om$-orthogonal intersections by our assumption on its intersection with $Q$. Hence, blowing the singular set back down symplectically, as described in section \ref{background}, gives our directed symplectic packing as required.
\end{proof}

It remains to show that such a curve $Q$ exists. For this we use the following result from \cite{mcop}, Corollary 1.2.17.

\begin{thm*}[McDuff-Opshtein]\label{curveexists}
Suppose $P$ closed balls of sizes $p^j$ can be embedded in $\P^2$ and let $\hat \P^2$ denote the corresponding blow-up with exceptional divisors $F^j$. Let $\cs$ be a singular set in $\hat \P^2$. Then the class $$A=N(L- \sum_j p^j [F^j])$$ has a connected embedded symplectic representative $Q$ which intersects $\cs$ positively and transversally provided $A \cdot S^j \ge 0$ for all component curves $S^j$ of $\cs$.
\end{thm*}

Applying this to our class $A$, the intersection of $A$ with the $S_{i,j}$ is proportional to their area, and the intersection with $\hat C$ is positive precisely if our area hypotheses are fulfilled. Hence Theorem \ref{curveexists} produces a curve in our class. The curve has positive self-intersection because this self-intersection is the volume of the blow-up, following section \ref{background}, of a $\delta$ neighborhood of the original disjoint embedded ellipsoids.

This completes the proof of Lemma \ref{le:directed}. \cqfd


\subsection{Proof of theorem \ref{thm:irrstabell}}\label{section33}

By approximating a general ellipsoid by a rational ellipsoid, and using the uniqueness up to isotopy of ball packings of an ellipsoid,
see \cite{mcduff3,mcop, mcduff4}, we see that to establish the strong packing stability property it suffices to show the following.

\begin{lemma}
$\forall a\in \Q$, $E(1,a)$ has the strong packing stability.
\end{lemma}
\noindent {\it Proof:} Assume without loss of generality that $a>1$, and consider the standard toric embedding $\phi_0$ of $E(a-1,a)$ into $\P^2(a)$, whose image has moment polytope drawn in Figure \ref{fig:twoell}. Here $\P^2(a)$ denotes complex projective space with lines scaled to have area $a$. Notice that the complement of the closure of this ellipsoid is exactly $E(1,a)$.

\begin{figure}[h!]
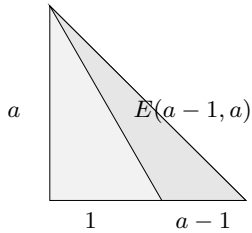

\begin{center}
\input twoellinball.pstex_t
\caption{$\P^2$ as union of two ellipsoids.}
\label{fig:twoell}
\end{center}\vspace{-,5cm}
\end{figure}

\paragraph{Step I: we produce a packing of $\P^2(a)\priv\tau E_0(a-1,a)$ for $\tau<1$.} Consider the sequence $(w_i)$ associated by Theorem \ref{mcduffthm}
to the ellipsoid $E(a-1,a)$. Since the closure of $E(a-1,a)$ embeds into $\P^2(a)$,  so does
$\sqcup B^4(w_i)$. Consider the symplectic form $\hat \om$ obtained by blowing-up these balls. By Lemma \ref{le:ratstabpack}, $(\P^2,\hat \om)$ has the packing stability. Thus, for numbers $a_i<\Lambda(\hat \om)$ with $\sum \frac{a_i^2}{2}<\vol(\hat \P^2_p,\hat \om)$, we have $\sqcup B(a_i)\omhra (\hat \P^2,\hat \om)$.
It is clear (either by Lemma \ref{le:directed} or by direct analysis in this case), that these balls also embed into the complement of the exceptional divisors in $\hat \P^2_p$. Blowing down the exceptional divisors, we then get a packing of $\P^2(a)$ by $\sqcup B(w_i)\sqcup B(a_i)$, hence by Theorem \ref{mcduffthm} a packing of $\P^2(a)$ by $E(a-1,a)\sqcup  B(a_i)$.

Fix now $\tau<1$, consider the induced packing of $\P^2(a)$ by $\tau E(a-1,a)\sqcup  B(a_i)$, and call $L_1,L_2$ the symplectic lines in $\P^2$ which correspond to the lines $\{w=0\}$ and $\{\pi(|z|^2+|w|^2)=a\}$ on the toric model. Since $\tau a$ and $\tau(a-1)$ are less than $a$, Lemma \ref{le:directed} ensures that there exists a symplectic embedding $\phi$ of $\tau E(a-1,a) \sqcup  B(a_i)$ such that $L_2$ intersects $\tau E(a-1,a)$ along its big axis (of size $\tau a$) and avoids all the balls, while $L_1$ intersects $\tau E(a-1,a)$ along its small axis and avoids the balls. Arguing as in \cite{moi5}, Corollary $2.2$, we can isotope $\phi(\tau E(a-1,a))$ to the standard embedding $\phi_0$, keeping $L_1,L_2$ fixed. As a result, we therefore get a packing of $\psi:\sqcup B(a_i)\omhra X:=\P^2(a)\priv(L_1\cup L_2\cup \phi_0(\tau E(a-1,a)))$.

\paragraph{Step II: we get a packing of $E(1,a)$.}
Consider now the contracting Liouville vector field $X_0$ on $\P^2\priv L_2$ which corresponds to the radial vector field $-\sum r_i\frac{\partial}{\partial r_i}=-\sum R_i\frac{\partial}{\partial R_i}$ on $B^4(a)=\P^2(a)\priv L_2$. A straightforward computation shows that the flow $\Phi^{f(\tau)}_{X_0}\circ \psi$ provides embeddings of symplectic balls of size $[1-f(\tau)]a_i$ into $E(1,a)$, for some $f(\tau) $ which tends to $0$ as $\tau \to 1$. We therefore get packings of $E(1,a)$ by balls of size arbitrarily close to $a_i$.

As ball packings are unique up to isotopy, see again \cite{mcop, mcduff4}, a limiting argument implies that we can actually get a packing of the whole open balls. \cqfd

\begin{rk} The previous proof can be roughly described as follows: an ellipsoid is the complement of an ellipsoid in $\P^2$, the latter being obtained from some specific ball packing in $\P^2$.  Thus packing the first ellipsoid amounts to packing $\P^2$ with more balls. This argument is completely  similar to McDuff's proof of  Hofer's conjecture \cite{mcduff5}.  The difference between the two approaches lies in the justification of this "thus" in the second sentence of the remark. McDuff refers to her inner/outer approximation technique in \cite{mcduff4}, while we base our proof on lemma \ref{le:directed}.
\end{rk}
\begin{lemma}
The function $\Lambda(a):=\Lambda(E(1,a))$ is locally bounded.
\end{lemma}
\noindent{\it Proof: }  We need to rework slightly the previous proof. Since we can approximate any ellipsoid by rational ellipsoids from the inside, we need only prove the statement for $a\in \Q$. So let us fix $a,a'\in \Q$, and consider the sequences $(w_1,\dots,w_p)$, $(w'_1,\dots,w_{p'}')$ associated to the ellipsoids $E(a-1,a)$ and $E(a'-1,a')$ by Theorem \ref{mcduffthm}. By Lemma \ref{le:stabw}, if we fix $a$ and consider $a'$ sufficiently close to $a$, then $p'>p$ and we can ensure that the $w_i'$ are arbitrarily close to the $w_i$ for $i\leq p$ and to $0$ for $i>p$.

Let $\hat \om'$ be the symplectic form obtained by blowing-up a packing $B(w_1')\sqcup\dots\sqcup B(w_p')$ of $\P^2$ (here we indeed mean $w_p'$ and not $w_{p'}'$). Lemma \ref{le:ratstabpack} ensures that there exists $\Lambda_{\hat \om'}>0$ such that  for any collection of numbers $a_i$ which abide by the volume constraint and $a_i<\Lambda(\hat \om')$ we have
$$
\sqcup B(a_i)\hra (\hat \P^2_p,\hat \om').
$$
Moreover, $\Lambda(\hat \om')>\frac{d_{[\hat \om]}}{2}$ provided that $(w_1',\dots,w_p')$ are close enough to $(w_1,\dots,w_p)$ - this is the case provided $a'$ is close enough to $a$, which we assume henceforth. We recall from Lemma \ref{le:dom} that the quantity $\frac{d_{[\hat \om]}}{2}$ is bounded away from $0$.

Assuming $a'$ even closer to $a$, we can also ensure that $w_i'<\frac{d_{[\hat \om]}}2$ for all $i>p$. Thus, $\P^2(1)$ has a packing by
$\sqcup B(w_i')\sqcup B(a_i)$,
for any collection of $(a_i)$ such that the volume constraint is satisfied and $a_i<\frac{d_{[\hat \om]}}{2}$. Arguing as in the previous proof, we get a packing of $E(1,a')$ ($\simeq \P^2(1)\priv E(a'-1,a')$) by the balls of size $a_i$, as long as the $a_i$ are less than $\frac{d_{[\hat \om]}}{2}$ together with the volume constraint. Hence $\Lambda(a')$ is bounded from below for $a'$ close to $a$ and the proof is complete. \cqfd

\section{Packing stability of pseudo-balls  $T(a,b,\alpha,\beta)$}\label{packpseudo}

We recall the definition of a pseudo-ball as the intersection of two ellipsoids.

\begin{definition}\label{def:ball-like}
A pseudo-ball is a domain $T(a,b,\alpha,\beta)\subset \C^2$ with $a>\alpha,b>\beta$, $a,b< \alpha+\beta$, defined by
$$
\begin{array}{l}
T(a,b,\alpha,\beta):=\{(z,w)\in \C^2\; |\; (\pi |z|^2, \pi |w|^2)\in Q(a,b,\alpha,\beta)\},\text{ where}\\
Q(a,b,\alpha,\beta):=\conv\langle (0,0), (0,a), (b,0), (\alpha,\beta) \rangle\subset \R^2.
\end{array}
$$
\end{definition}
Our proof of the packing stability for ellipsoids used the fact that an ellipsoid is always obtained by excising an ellipsoid from $\P^2$. We notice that a pseudo-ball is obtained by excising two ellipsoids (see Figure \ref{fig:ball-like}).

\begin{figure}[h!]
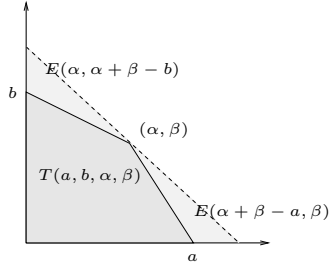

\begin{center}
\input pseudo-ball.pstex_t
\caption{$\P^2$ as union of two ellipsoids and the pseudo-ball.}
\label{fig:ball-like}
\end{center}
\end{figure}

To be precise we have the following.

\begin{lemma}
$T(a,b,\alpha,\beta)=\P^2(\alpha+\beta)\priv( E\cup E')$, where $E,E'$ are the toric ellipsoids with equations
$$
\begin{array}{l}
\pi(E'):=\conv\langle (0,b), (0,\alpha+\beta),(\alpha,\beta)\rangle,\\
\pi(E):=\conv \langle (a,0),(\alpha+\beta,0),(\alpha,\beta) \rangle.
\end{array}
$$
\end{lemma}

Moreover, arguing again as in \cite{moi5}, we see that these domains are the complement in $\P^2$ of {\it any} such ellipsoids which are in a good position with the coordinate axis.

\begin{prop}\label{prop:Tpackstab}
Let $L_1,L_2,L_3$ be three lines in $\P^2(\alpha+\beta)$. Let $E$ be an embedding of the ellipsoid $E(\alpha+\beta-a,\alpha)$ into $\P^2$ whose axis are discs $D_1,D_3$ in $L_1,L_3$ of area $\alpha+\beta-a, \alpha$ respectively. Let $E'$ be an embedding of the ellipsoid $E(\alpha+\beta-b,\beta)$ whose axis are discs $D_2',D'_3$ in $L_2,L_3$ of area $\alpha+\beta-b,\beta$ respectively (notice that the discs $D_3,D_3'$ must cover $L_3$). Then, $\P^2(\alpha+\beta)\priv (E\cup E')$ is  symplectomorphic to $T(a,b,\alpha,\beta)$.
\end{prop}

With this in hand, the proof of Theorem \ref{thm:irrstabpseudo} is completely analogous to the packing stability of the ellipsoids, and so we do not repeat it here.


\section{Packing stability of irrational closed symplectic $4$-manifolds}\label{mainthm}

The object of this section is the proof of Theorem \ref{thm:irrstabclosed}. We begin with the general decomposition of symplectic $4$-manifolds, then derive packing stability as a consequence of this and our results for open manifolds.


\subsection{Decomposition of symplectic $4$-manifolds with irrational symplectic class}\label{mainone}
We have already referred to the fact that a rational symplectic manifold is always fully packed by a single ellipsoid. In particular, strong packing stability for ellipsoids implies strong packing stability for all rational symplectic manifolds. Since Conjecture \ref{conj1} remains open packing stability for ellipsoids does not necessarily imply the same for irrational manifolds, but we now recall some results obtained in \cite{moi6} which allow us to split a symplectic manifold with irrational symplectic class into finitely many standard pieces, namely ellipsoids and pseudo-balls. As before, the packing stability property for an irrational symplectic manifold will follow from
the packing properties of these standard pieces.

\begin{definition} A singular polarization of a symplectic $4$-manifold $(M,\om)$ is a weighted multi-curve ${\bf \Sigma}=(\Sigma_i,\alpha_i)$, where:
\begin{itemize}
\item[\sbull] $\Om:=[\om]=\sum \alpha_i \pd([\Sigma_i])$,
\item[\sbull] Each $\Sigma_i$ is an embedded symplectic curve,
\item[\sbull] $\Sigma_i$ intersects $\Sigma_j$ transversally and positively,
\item[\sbull] $\Sigma_i\cap \Sigma_j\cap \Sigma_k=\emptyset$.
\end{itemize}
\end{definition}
\begin{thm*}[\cite{moi6,moi7}]
Any closed symplectic $4$-manifold $M$ has a singular polarization $({\bf \Sigma,\alpha})$. The complement of $\cup \Sigma_i$ in $M$ is endowed with a tame Liouville form $\lambda$ with residues $\alpha_i$ at $\Sigma_i$.
\end{thm*}
The precise definition of the residues is not so important here, and can be found in \cite{moi7}. The vector field dual to $\lambda$ is a (contracting) Liouville vector field $X_\lambda$, which points outwards along $\cup \Sigma_i$. It is therefore forward complete. The basin of repulsion of a subset $X\subset\cup \Sigma_i$ is defined as
$$
\cb_\lambda(X):=\{p\in M\;|\; \exists t_0<0, \;\Phi_{X_\lambda}^{t}(p) \underset{t\to t_0}\lra X\}.
$$
As we see next proposition, the basins of repulsion  of well-chosen subsets fully pack $M$, and are standard.

\begin{prop} \label{prop:attrbasins} The basin of attraction of a disc $\D(a_i)\subset \Sigma_i$ is an embedded ellipsoid $E(a_i,\alpha_i)$. The basin of attraction of a cross $\D(a_i)\cup\D(a_j)\subset \Sigma_i\cup \Sigma_j$ (the two discs intersect at exactly one point) is an embedded pseudo-ball $T(a_i,a_j,\alpha_j,\alpha_i)$. Finally, if a family of discs and crosses cover the polarization up to area $0$, their basin of attraction cover $M$ up to volume $0$.
\end{prop}
\noindent {\it Proof:} Since the flow $\Phi^t_{X_\lambda}$ is forward complete and contracts the symplectic form, it is obvious that the complement of the basin of repulsion of the whole polarization has zero volume. On the other hand, the basin of repulsion of a subset of the polarization with zero-area has zero volume. We conclude that if discs and crosses cover the polarization up to area $0$, their basins of repulsion cover $M$ up to volume $0$ (we refer to \cite{moi7} for more details).

It is moreover explained in this paper that these basins of repulsion can be computed in {\it any} polarization, of any manifold, provided that these polarizations have the same weights. In particular, the basin of repulsion of $\D(a_i)\subset \Sigma_i$ is symplectomorphic to the basin of repulsion of $\D(a_i)\times\{0\}\subset \C^2$, with Liouville form $\alpha_id\theta_2-r_1^2d\theta_1-r_2^2d\theta_2$. (Note that the Liouville form has residue $\alpha_i$ about the $z_1$ axis.) A computation done in \cite{moi6} (Proposition 3.5) shows that this basin of repulsion is the ellipsoid $\ce(a_i,\alpha_i)$. The basin of repulsion of a cross $\D(a_i)\times \D(a_j)\subset \Sigma_i\times\Sigma_j$ is symplectomorphic to the basin of repulsion of $\D(a_i)\times\{0\}\cup \{0\}\times \D(a_j)$, for the Liouville form $\alpha_id\theta_2+\alpha_jd\theta_1 -R_1d\theta_1-R_2d\theta_2$, where $R_i=r_i^2$. (The Liouville form now has residues $\alpha_i$ about the $z_1$ axis and $\alpha_j$ about the $z_2$ axis.) The dual vector field is
$$
X_\lambda= (\alpha_j-R_1)\frac\partial{\partial R_1}+(\alpha_i-R_2)\frac\partial{\partial R_2}.
$$
Therefore, the trajectories are
$$
\Phi^t_{X_\lambda}(R_1,\theta_1,R_2,\theta_2)=(\alpha_j+(R_1-\alpha_j)e^{-t},\theta_1,\alpha_i+(R_2-\alpha_i)e^{-t},\theta_2).
$$
Thus, this vector field commutes with the classical Hamiltonian toric action on $\C^2$, and
in toric coordinates, the trajectories are straight segments whose common endpoint as $t$ goes to $+\infty$ is the point $(\alpha_j,\alpha_i)$. Since the cross is the preimage by the toric projection of $[0,a_i]\times\{0\}\cup \{0\}\times[0,a_j]\subset \R^2$, its basin of repulsion is the preimage of the convex hull of $\{(0,0),(a_i,0),(0,a_j),(\alpha_j,\alpha_i)\}$, by the toric projection, and hence is indeed $T(a_i,a_j,\alpha_j,\alpha_i)$.\cqfd

\subsection{Proof of theorem \ref{thm:irrstabclosed}}\label{maintwo}
Let $(M,\om)$ a closed symplectic  manifold, on which we lose no generality by assuming its volume to be $1$.

Fix a singular polarization ${\bf \Sigma}=\{(\Sigma_i,\alpha_i),i=1\dots l\}$ such that $\ca_\om(\Sigma_i)\geq 10 \alpha_j$ $\forall i,j$, and $\Sigma_i\cap \Sigma_j\neq \emptyset$ $\forall (i,j)$. In fact, the construction proposed in \cite{moi6} provides singular polarizations that automatically satisfy these two constraints: they intersect a lot, the area of the curves are very large, and the $\alpha_i$ are very small. Now consider points $x_i\in \Sigma_i\cap \Sigma_{i+1}$ (with $\Sigma_{l+1}=\Sigma_1$), and  discs $D_i,D_{i,i-1},D_{i,i+1}\subset \Sigma_i$ with areas $\ca_i,\ca_{i,i-1},\ca_{i,i+1}$ respectively, such that:
\begin{enumerate}
\item The discs are disjoint, and their union has full area in $\Sigma_i$ $\forall i$,
\item $x_{i-1}\in D_{i,i-1}$, $x_i\in D_{i,i+1}$,
\item
$\ca_{i,i\pm 1}\in]\alpha_{i\pm 1},\alpha_i+\alpha_{i\pm 1}[$.
Thus
the domains
$T(\ca_{i,i+1},\ca_{i+1,i},\alpha_{i+1},\alpha_i)$ are pseudo-balls (see figure \ref{fig:decomp}).
\end{enumerate}
Condition $3$ is easily achieved by choosing the discs inductively  so that these conditions hold (choose first $D_1$, then $D_{1,2}$, $D_2$, $D_{2,3}$ \dots). Note that in order to satisfy condition $1$, the area $\ca_{1,2}$ is determined by $\ca_1$; the areas $\ca_{2,1}$ and $\ca_2$ determine $\ca_{2,3}$ and so on.

\begin{figure}[h!]
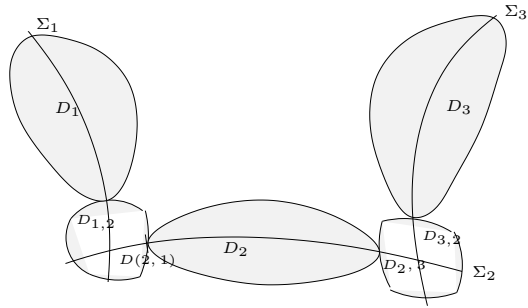

\begin{center}
\input decomp.pstex_t
\caption{Decomposition of a closed $4$-manifold into ellipsoids and pseudo-balls}
\label{fig:decomp}
\end{center}
\end{figure}

By Proposition \ref{prop:attrbasins} the
$\ce(\ca_i,\alpha_i)$ and the
 $T(\ca_{i,i+1},\ca_{i+1,i},\alpha_{i+1},\alpha_i)$ give a full packing of $M$.
Call $m$ its number of pieces.

Now, by Theorem \ref{thm:irrstabell} and Proposition \ref{prop:Tpackstab}, there exists a $\Lambda$ and $\eps_0$ such that all ellipsoids $E(\ca_i+\eps_i,\alpha_i)$, and all pseudo-balls  $T(\ca_{i,i+1}+\eps_{i,i+1},\ca_{i+1,i}+\eps_{i+1,i},\alpha_i,\alpha_{i+1})$ have the strong packing stability, with constant $\Lambda$, for all $\eps_i,\eps_{i,j}<\eps_0$. Moreover, there exists a constant $\delta$ and a sequence $\eps'_1 < \eps'_2 < \dots < \eps'_m=\eps_0$ such that we can achieve any volume within $\delta$ of the original volume of one of our ellipsoids by changing the $\eps_i$ parameter by less than $\eps'_1$; and if one parameter of any of our pseudo-balls is changed by less than $2\eps'_l$ then we can change the second parameter by less than $\eps'_{l+1}$ to achieve any volume within $\delta$ of the original.

We claim that $M$ satisfies the packing stability with constant $\Lambda'=\min(\Lambda, \sqrt{2\delta} )$. (Note that a ball of capacity $\sqrt{2\delta}$ has volume $\delta$.) Indeed, suppose we have a sequence of balls, each of capacity less that $\Lambda'$. If the total volume of the balls is less than $1$ we may include some additional ones. Then we can partition the balls into $m$ subsets such that the sum of the volumes of the balls in the $i$th block lies within $\delta$ of the volume of the $i$th piece of our decomposition of $M$.

Now we claim that we can perturb the discs $D_i,D_{i,j}$ to discs $D_i',D_{i,j}'$ which verify the four conditions listed above but are such that the volume of the associated piece of our decomposition is equal to the total volume of the corresponding subset of balls. We do this in order. First we perturb $\ca_1$ by less than $\eps_1$ to get an $\ca'_1$ such that $E(\ca'_1,\alpha_1)$ has the volume of the first subset of balls. The area $\ca'_{1,2}$ is then determined, but will be changed by less than $\eps_1$. We can then perturb $\ca_{2,1}$ by less than $\eps_2$ to get a new pseudo-ball $T(\ca'_{1,2}, \ca'_{2,1}, \alpha_1, \alpha_2)$ whose volume is equal to the volume of the second subset of balls. Carrying on in this way we get the volumes as claimed.

Finally, as the perturbed pieces still have stability constant less than $\Lambda$ they can be fully filled by the corresponding subset of balls, and as a consequence the balls fully pack $M$ as required. \hfill $\square$



{\footnotesize
\bibliographystyle{alpha}
\bibliography{bib3}
}
\end{document}